\documentclass[11pt,a4paper]{article}
\usepackage{graphicx}
\usepackage{amsmath}

\newtheorem{theorem}{Theorem}

\newtheorem{definition}[theorem]{Definition}

\newtheorem{lemma}[theorem]{Lemma}

\newenvironment{proof}[1][Proof]{\textbf{#1.} }{\ \rule{0.5em}{0.5em}}

\newcommand{\calF}{{\cal F}}

\newcommand{\ep}{\varepsilon}

\begin{document}

\title{Equilibrium in Two-Player Non-Zero-Sum Dynkin Games in Continuous Time%
\thanks{%
The results presented in this paper were proven while the authors attended
the workshop on ``Repeated Games and Differential Games'', organized by Marc
Quincampoix and Sylvain Sorin in November 2008, Roscoff, France. We thank
Said Hamad\`ene for his assistance and for his helpful comments. The work of
Solan was supported by the Israel Science Foundation, Grant 212/09.}}
\author{Rida Laraki\thanks{%
CNRS \textit{and} Laboratoire d'Econom\'{e}trie de l'Ecole Polytechnique,
91128, Palaiseau, France. Email: rida.laraki@polytechnique.edu} \ and Eilon
Solan\thanks{%
Corresponding author: School of Mathematical Sciences, Tel Aviv University,
Tel Aviv 69978, Israel. Email: eilons@post.tau.ac.il}}
\maketitle

\begin{abstract}
We prove that every two-player non-zero-sum Dynkin game in continuous time
admits an $\varepsilon$-equilibrium in randomized stopping times. We provide
a condition that ensures the existence of an $\varepsilon$-equilibrium in
non-randomized stopping times.
\end{abstract}

\bigskip

\bigskip

%\bigskip\bigskip
%\noindent \textbf{AMS Classification: } 91A05
\newpage

\section{Introduction}

Dynkin games (Dynkin, 1969) serve as a model of optimal stopping. These
games were applied in various setups, including wars of attrition (see,
e.g., Maynard Smith (1974), Ghemawat and Nalebuff (1985) and Hendricks et
al. (1988)), pre-emption games (see, e.g., Fudenberg and Tirole (1991,
section 4.5.3)), duels (see, e.g., Blackwell (1949), Bellman and Girshick
(1949), Shapley (1951), Karlin (1959), and the survey by Radzik and Raghavan
(1994)), and pricing of options (Kifer (2000), Hamad\`{e}ne (2006)).

The existence of a value in randomized strategies in Dynkin games, in its
general setting, has been settled only recently (see Rosenberg, Solan and
Vieille (2001) for discrete time games, and Laraki and Solan (2005), for
continuous time games). The existence of an $\varepsilon $-equilibrium in
randomized strategies in non-zero-sum games has been proven for two-player
games in discrete time (Shmaya and Solan, 2004), and for games in continuous
time under certain conditions (see, e.g., Laraki, Solan and Vieille, 2005).

In the present paper we prove that every two-player non-zero-sum Dynkin game
in continuous time admits an $\varepsilon$-equilibrium in randomized
strategies, for every $\varepsilon > 0$. We further show how such an
equilibrium can be constructed, and we provide a condition under which there
exists an $\varepsilon$-equilibrium in non-randomized strategies. Rather
than using the Snell envelope, as, e.g., in Hamad\`{e}ne and Zhang (2010),
our technique is to use results from zero-sum games.

We note that three-player Dynkin games in continuous time may fail to admit
an $\varepsilon $-equilibrium in randomized strategies, even if the payoff
processes are constant (Laraki, Solan and Vieille, 2005, Section 5.2). Thus,
our result completes the mapping of Dynkin games in continuous time that
admit an $\varepsilon$-equilibrium in randomized stopping times.

The paper is organized as follows. The model and the main results appear in
Section \ref{sec model}. In Section \ref{sec0} we review known results
regarding zero-sum games that are then used in Section \ref{sec1} to prove
the main theorem.

\section{Model and Results}

\label{sec model}

Let $(\Omega ,\mathcal{A},P)$ be a probability space, and let $\mathcal{F}=(%
\mathcal{F}_{t})_{t\geq 0}$ be a filtration in continuous time that
satisfies ``the usual conditions''. That is, $\mathcal{F}$ is right
continuous, and $\mathcal{F}_{0}$ contains all $P$-null sets: for every $%
B\in \mathcal{A}$ with $P(B)=0$ and every $A\subseteq B,$ one has $A \in
\mathcal{F}_{0}$. All stopping times in the sequel are w.r.t.~the filtration
$\mathcal{F}$.

Denote $\mathcal{F}_{\infty }:=\vee _{t\geq 0}\mathcal{F}_{t}$. Assume
without loss of generality that $\mathcal{F}_{\infty }=\mathcal{A}.$ Hence $%
(\Omega ,\mathcal{A},P)$ is a complete probability space.

Let $(X_{i},Y_{i},Z_{i})_{i=1,2}$ be uniformly bounded $\mathcal{F}$-adapted
real-valued processes,\footnote{%
Our results hold for the larger class of $\mathcal{D}$ payoff processes
defined by Dellacherie and Meyer, 1975, \S II-18. This class contains in
particular integrable processes.} and let $(\xi _{i})_{i=1,2}$ be two
bounded real-valued $\mathcal{F}_{\infty }$-measurable functions. In the
sequel we will assume that the processes $(X_{i},Y_{i},Z_{i})_{i=1,2}$ are
right continuous.

\begin{definition}
A \emph{two-player non-zero-sum Dynkin game} over $(\Omega ,\mathcal{A},P,%
\mathcal{F})$ with payoffs \newline
$(X_{i},Y_{i},Z_{i},\xi _{i})_{i=1,2}$ is the game with player set $%
N=\{1,2\} $, the set of pure strategies of each player is the set of
stopping times, and the payoff function of each player $i\in \{1,2\}$ is:
\begin{equation}
\gamma _{i}(\lambda _{1},\lambda _{2}):={\mathbf{E}}\left[ X_{i}(\lambda
_{1}){\mathbf{1}}_{\{\lambda _{1}<\lambda _{2}\}}+Y_{i}(\lambda _{2}){%
\mathbf{1}}_{\{\lambda _{2}<\lambda _{1}\}}+Z_{i}(\lambda _{1}){\mathbf{1}}%
_{\{\lambda _{1}=\lambda _{2}<\infty \}}+\xi _{i}{\mathbf{1}}_{\{\lambda
_{1}=\lambda _{2}=\infty \}}\right] ,  \label{equ 1}
\end{equation}%
where $\lambda _{1}$ and $\lambda _{2}$ are the stopping times chosen by the
two players respectively.
\end{definition}

In words, the process $X_{i}$ represents the payoff to player $i$ if player
1 stops before player 2, the process $Y_{i}$ represents the payoff to player
$i$ if player 2 stops before player 1, the process $Z_{i}$ represents the
payoff to player $i$ if the two players stop simultaneously, and the
function $\xi _{i}$ represents the payoff to player $i$ if no player ever
stops.

The game is \emph{zero-sum} if $X_{1}+X_{2}=Y_{1}+Y_{2}=Z_{1}+Z_{2}=\xi
_{1}+\xi _{2}=0$.

In non-cooperative game theory, a randomized strategy is a probability
distribution over pure strategies, with the interpretation that at the
outset of the game the player randomly chooses a pure strategy according to
the probability distribution given by the randomized strategy, and uses it
along the game. In the setup of Dynkin games in continuous time, a
randomized strategy is a randomized stopping time, which is defined as
follows.

\begin{definition}
\label{def2} A \emph{randomized stopping time} for player $i$ is a
measurable function $\varphi _{i}:[0,1]\times \Omega \rightarrow \lbrack
0,+\infty ]$ such that the function $\varphi _{i}(r,\cdot ):\Omega
\rightarrow \lbrack 0,+\infty ]$ is a stopping time for every $r\in \lbrack
0,1]$ (see Aumann (1964)).
\end{definition}

Here the interval $[0,1]$ is endowed with the Borel $\sigma $-field. For
strategically equivalent definitions of randomized stopping times, see Touzi
and Vieille (2002). The interpretation of Definition \ref{def2} is that
player $i$ chooses $r$ in $[0,1]$ according to the uniform distribution and
then stops at the stopping time $\varphi _{i}(r,\cdot )$. Throughout the
paper, the symbols $\lambda $, $\mu $ and $\tau $ stand for stopping times,
and $\varphi $ and $\psi $ stand for randomized stopping times.

The \emph{expected payoff} for player $i$ that corresponds to a pair of
randomized stopping times $(\varphi _{1},\varphi _{2})$ is:
\begin{equation*}
\gamma_i (\varphi _{1},\varphi _{2}):=\int_{[0,1]^{2}}\gamma_i (\varphi
_{1}(r,\cdot),\varphi _{2}(s,\cdot))\ dr\ ds, \ \ \ \ \ i=1,2.
\end{equation*}

In the sequel we will also consider the expected payoff at a given time $t$.
We therefore define for every $t \geq 0$ and every pair of randomized
stopping times $\varphi_1,\varphi_2 \geq t$:
\begin{equation}
\gamma _{i}(\varphi _{1},\varphi _{2}\mid \mathcal{F}_t ):= {\mathbf{E}}%
[X_{i}(\varphi _{1}){\mathbf{1}}_{\{\varphi _{1}<\varphi
_{2}\}}+Y_{i}(\varphi _{2}){\mathbf{1}}_{\{\varphi _{2}<\varphi
_{1}\}}+Z_{i}(\varphi _{1}){\mathbf{1}}_{\{\varphi _{1}=\varphi _{2}<\infty
\}}+\xi _{i}{\mathbf{1}}_{\{\varphi _{1}=\varphi _{2}=\infty \}}\mid
\mathcal{F}_t].  \label{equ 1b}
\end{equation}

A pair of randomized stopping times $(\varphi_1^*,\varphi_2^*)$ is an $%
\varepsilon$-equilibrium if no player can profit more than $\varepsilon $ by
deviating from $\varphi_i^*$.

\begin{definition}
Let $\varepsilon \geq 0$. A pair of randomized stopping times $(\varphi
_{1}^{\ast },\varphi _{2}^{\ast })$ is an \emph{$\varepsilon $-equilibrium}
if for every two randomized stopping times $\varphi _{1},\varphi _{2} $ the
following inequalities hold:
\begin{equation}
\gamma _{1}(\varphi _{1},\varphi _{2}^{\ast })\leq \gamma _{1}(\varphi
_{1}^{\ast },\varphi _{2}^{\ast })+\varepsilon ,  \label{equ 2}
\end{equation}%
and
\begin{equation}
\gamma _{2}(\varphi _{1}^{\ast },\varphi _{2})\leq \gamma _{2}(\varphi
_{1}^{\ast },\varphi _{2}^{\ast })+\varepsilon.  \label{equ 2a}
\end{equation}
\end{definition}

Because of the linearity of the payoff function, Eqs. (\ref{equ 2}) and (\ref%
{equ 2a}) hold for every randomized stopping time $\varphi_1$ and $\varphi_2$
respectively as soon as they hold for non-randomized stopping times.

Our goal in this paper is to prove the existence of an $\varepsilon $%
-equilibrium in two-player non-zero-sum games, and to construct such an $%
\varepsilon $-equilibrium.

Suppose that a player wants to stop at the stopping time $\lambda$, but he
would like to mask the exact time at which he stops (for example, so that
the other player cannot stop at the very same moment as he does). To this
end, he can stop at a randomly chosen time in a small interval $%
[\lambda,\lambda+\delta]$, and, since the payoff processes are right
continuous, he will not lose (or gain) much relative to stopping at time $%
\lambda $. This leads us to the following class of simple randomized
stopping times that will be extensively used in the sequel.

\begin{definition}
\label{def simple} A randomized stopping time $\varphi $ is \emph{simple} if
there exist a stopping time $\lambda$ and a $\mathcal{F}_\lambda$-measurable
non-negative function $\delta \geq 0$, such that for every $r\in \lbrack
0,1] $ one has $\varphi (r,\cdot)=\lambda +r\delta $. The stopping time $%
\lambda$ is called the \emph{basis} of $\varphi$, and the function $\delta$
is called the \emph{delay} of $\varphi$.
\end{definition}

Since $\varphi (r,\cdot)\geq \lambda $ and $\varphi (r,\cdot)$ is $\mathcal{F%
}_{\lambda } $-measurable, by Dellacherie and Meyer (1975, \S IV-56), $%
\varphi (r,\cdot)$ is a stopping time for every $r\in \lbrack 0,1].$
Consequently, $\varphi $ is indeed a randomized stopping time.

Definition \ref{def simple} does not require that $\lambda$ is finite:%
\footnote{%
A statement \emph{holds} on a measurable set $A$ if and only if the set of
points in $A$ that do not satisfy the statement has probability 0.} on the
set $\{\lambda = \infty\}$ we have $\varphi(r,\cdot) = \infty$ for every $r
\in [0,1]$. On the set $\{\delta = 0\}$ the randomized stopping time $\varphi
$ that is defined in Definition \ref{def simple} stops at time $\lambda$
with probability 1. On the set $\{\delta > 0\}$ the stopping time is
``non-atomic'' yet finite, and in particular for every stopping time $\mu$
we have ${\mathbf{P}}(\{\delta > 0\} \cap \{\varphi = \mu\}) = 0$.

We now state our main results.

\begin{theorem}
\label{theorem1} Every two-player non-zero-sum Dynkin game with
right-continuous and uniformly bounded payoff processes admits an $%
\varepsilon$-equilibrium in simple randomized stopping times, for every $%
\varepsilon > 0$.
\end{theorem}

Moreover, the delay of the simple randomized stopping time that constitute
the $\varepsilon$-equilibrium can be arbitrarily small.

Theorem \ref{theorem1} was proved by Laraki and Solan (2005) for two-player
\emph{zero-sum} games. Our proof heavily relies on the results of Laraki and
Solan (2005), and we use $\varepsilon$-equilibria in zero-sum games to
construct an $\varepsilon$-equilibrium in the non-zero-sum game.

Under additional conditions on the payoff processes, the $\varepsilon$%
-equilibrium is given in non-randomized stopping times.

\begin{theorem}
\label{theorem2} Under the assumptions of Theorem \ref{theorem1}, if $%
Z_{1}(t)\in co\{X_{1}(t),Y_{1}(t)\}$ and $Z_{2}(t)\in co\{X_{2}(t),Y_{2}(t)\}
$ for every $t \geq 0$, then the game admits an $\varepsilon $-equilibrium
in non-randomized stopping times, for every $\varepsilon >0$.
\end{theorem}

Hamad\`{e}ne and Zhang (2010) proved the existence of a 0-equilibrium in
non-randomized stopping times under stronger conditions than those in
Theorem \ref{theorem2}, using the notion of Snell envelope of processes
(see, e.g., El-Karoui (1980) for more details).

The rest of the paper is devoted to the proofs of Theorems \ref{theorem1}
and \ref{theorem2}. We will assume w.l.o.g.~that the payoff processes are
bounded between 0 and 1.

\section{The Zero-Sum Case}

\label{sec0}

In the present section we summarize several results on zero-sum games, taken
from Laraki and Solan (2005), that will be used in the sequel, and prove
some new results on zero-sum games.

For every $t\geq 0$ denote
\begin{eqnarray}
v_{1}(t) &:=&{\mathrm{ess-sup}}_{\varphi _{1}\geq t}{\mathrm{ess-inf}}%
_{\lambda _{2}\geq t}{\mathbf{E}}{\large [}X_{1}(\varphi _{1}){\large {%
\mathbf{1}}}_{\{\varphi _{1}<\lambda _{2}\}}+Y_{1}(\lambda _{2}){\large {%
\mathbf{1}}}_{\{\lambda _{2}<\varphi _{1}\}}  \label{equ zero1} \\
&&\ \ \ +Z_{1}(\varphi _{1}){\mathbf{1}}_{\{\varphi _{1}=\lambda _{2}<\infty
\}}+\xi _{1}{\mathbf{1}}_{\{\varphi _{1}=\lambda _{2}=\infty \}}\mid
\mathcal{F}_{t}{\large ]},  \notag
\end{eqnarray}%
where the supremum is over all randomized stopping times $\varphi _{1}\geq t$%
, and the infimum is over all (non-randomized) stopping times $\lambda
_{2}\geq t$. This is the highest payoff that player 1 can guarantee in the
zero-sum Dynkin game $\Gamma _{1}(t)$ where the payoffs are those of player
1, player 1 is the maximizer, player 2 is the minimizer, and the game starts
at time $t$. Similarly, the highest payoff that player 2 can guarantee in
the zero-sum Dynkin game $\Gamma _{2}(t)$ where the payoffs are those of
player 2, player 2 is the maximizer, player 1 is the minimizer, and the game
starts at time $t$, is given by:
\begin{eqnarray}
v_{2}(t) &:=&{\mathrm{ess-sup}}_{\varphi _{2}\geq t}{\mathrm{ess-inf}}%
_{\lambda _{1}\geq t}{\mathbf{E}}{\large [}X_{2}(\lambda _{1}){\large {%
\mathbf{1}}}_{\{\lambda _{1}<\varphi _{2}\}}+Y_{2}(\varphi _{2}){\large {%
\mathbf{1}}}_{\{\varphi _{2}<\lambda _{1}\}}  \label{equ zero2} \\
&&\ \ \ +Z_{2}(\lambda_1){\mathbf{1}}_{\{\lambda _{1}=\varphi _{2}<\infty
\}}+\xi _{2}{\mathbf{1}}_{\{\lambda _{1}=\varphi _{2}=\infty \}}\mid
\mathcal{F}_{t}{\large ]}.  \notag
\end{eqnarray}

The next lemma, which is proved in Laraki and Solan (2005), states that $%
v_1(t)$ (resp. $v_2(t)$) is in fact the value of the zero-sum games $%
\Gamma_1(t)$ (resp. $\Gamma_2(t)$). This lemma is proved
in Laraki and Solan (2005) when $\mathcal{F}_t$ is the trivial $\sigma$%
-algebra. Its proof can be adapted to a general $%
\mathcal{F}_t$ (see the discussion in Appendix \ref{appendix}).

\begin{lemma}
\label{lemma8}
\begin{eqnarray}
v_{1}(t) &=&{\mathrm{ess-inf}}_{ \psi_{2}\geq t}{\mathrm{ess-sup}}_{
\lambda_{1}\geq t}{\mathbf{E}}{\large [}X_{1}(\lambda_{1}){\large {\mathbf{1}%
}}_{\{\lambda_{1}<\psi _{2}\}}{\large +}Y_{1}(\psi _{2}){\large {\mathbf{1}}}%
_{\{\psi _{2}<\lambda _{1}\}}  \label{equ zero1b} \\
&&\ \ \ +Z_{1}(\lambda _{1}){\mathbf{1}}_{\{\lambda _{1}=\psi _{2}<\infty
\}}+\xi _{1}{\mathbf{1}}_{\{\lambda _{1}=\psi _{2}=\infty \}}\mid \mathcal{F}%
_{t}{\large ]},  \notag
\end{eqnarray}%
and
\begin{eqnarray}
v_2(t) &=&{\mathrm{ess-inf}}_{\psi _{1}\geq t}{\mathrm{ess-sup}}_{\lambda
_{2}\geq t}{\mathbf{E}}{\large [}X_{2}(\psi _{1}){\large {\mathbf{1}}}%
_{\{\psi _{1}<\lambda _{2}\}}+ Y_{2}(\lambda _{2}){\large {\mathbf{1}}}%
_{\{\lambda _{2}<\psi _{1}\}}  \label{equ zero2b} \\
&&\ \ \ +Z_{2}(\psi _{1}){\mathbf{1}}_{\{\psi _{1}=\lambda _{2}<\infty
\}}+\xi _{2}{\mathbf{1}}_{\{\psi _{1}=\lambda _{2}=\infty \}}\mid \mathcal{F}%
_{t}{\large ]},  \notag
\end{eqnarray}%
where the infimum in (\ref{equ zero1b}) is over all randomized stopping
times $\psi_2 \geq t$ for player 2, the supremum in (\ref{equ zero1b}) is
over all (non-randomized) stopping times $\lambda_1 \geq t$ for player 1,
the infimum in (\ref{equ zero2b}) is over all randomized stopping times $%
\psi_1 \geq t$ for player 1, and the supremum in (\ref{equ zero2b}) is over
all (non-randomized) stopping times $\lambda_2 \geq t$ for player 2.
\end{lemma}

A stopping time $\varphi_1$ (resp. $\psi_1$) that achieves the supremum in (%
\ref{equ zero1}) (resp. infimum in (\ref{equ zero2b})) up to $\varepsilon$
is called an \emph{$\varepsilon$-optimal} stopping time for player 1 in $%
\Gamma_1(t)$ (resp. $\Gamma_2(t)$). Similarly, a stopping time $\varphi_2$
(resp. $\psi_2$) that achieves the supremum in (\ref{equ zero2}) (resp.
infimum in (\ref{equ zero1b})) up to $\varepsilon$ is called an \emph{$%
\varepsilon$-optimal} stopping time for player 2 in $\Gamma_2(t)$ (resp. $%
\Gamma_1(t)$).

%For every positive $\calF_t$-measurable function $\ep$
%we will denote by $\varphi_i(t,\ep)$ an $\ep$-optimal stopping time for player $i$ in the game $\Gamma_i(t)$,
%and by $\psi_i(t,\ep)$ an $\ep$-optimal stopping time of player $i$ in the game $\Gamma_{3-i}(t)$;
%thus, the strategy $\psi_i(t,\ep)$ is a punishment strategy against player $3-i$,
%as it ensures that his payoff will not exceed $v_{3-i}(t) + \ep$.

The proof of Laraki and Solan (2005, Proposition 7) can be adapted
to show that the value process is right continuous (see Appendix \ref{appendix}).

\begin{lemma}
\label{lemma zero sum} The process $(v_i(t))_{t \geq 0}$ is right
continuous, for each $i \in \{1,2\}$.
\end{lemma}

The following two lemmas provide crude bounds on the value process.

\begin{lemma}
\label{lemma vxy} For every $t \geq 0$ and each $i=1,2$ one has
\begin{equation*}
\min\{X_{i}(t),Y_{i}(t)\}\leq v_{i}(t)\leq \max \{X_{i}(t),Y_{i}(t)\}
\hbox{
on } \Omega.
\end{equation*}
\end{lemma}

\begin{proof}
We start by proving the left-hand side inequality for $i=2$. Let $%
\varepsilon >0$ be arbitrary, and let $\delta > 0$ be sufficiently small
such that
\begin{eqnarray}
{\mathbf{P}}(\sup_{\rho\in [0,\delta] }|X_{2}(t )-X_{2}(t+\rho)|>\varepsilon) \leq \varepsilon,
\label{equ 55a} \\
{\mathbf{P}}(\sup_{\rho\in [0,\delta] }|Y_{2}(t )-Y_{2}(t+\rho)|>\varepsilon) \leq \varepsilon.
\label{equ 56a}
\end{eqnarray}%
Such $\delta$ exists because the processes $X_2$ and $Y_2$ are right
continuous.

Let $\varphi _{2}$ be the simple randomized stopping time $\varphi
_{2}(r,\cdot )=t+r\delta $, and let $\lambda _{1}\geq t$ be any
non-randomized stopping time for player 1. The definition of $\varphi _{2}$
implies that the probability that $\lambda _{1}=\varphi _{2}$ is 0: ${%
\mathbf{P}}(\lambda _{1}=\varphi _{2})=0$. Moreover, $\varphi _{2}<\infty $.
Therefore
\begin{equation*}
\gamma _{2}(\lambda _{1},\varphi _{2}\mid \mathcal{F}_{t})={\mathbf{E}}%
[X_{2}(\lambda _{1})\mathbf{1}_{\{\lambda _{1}<\varphi _{2}\}}+Y_{2}(\varphi
_{2})\mathbf{1}_{\{\varphi _{2}<\lambda _{1}\}}\mid \mathcal{F}_{t}].
\end{equation*}%
By (\ref{equ 55a}) and (\ref{equ 56a}) this implies that
\begin{equation*}
{\mathbf{P}}(\gamma _{2}(\lambda _{1},\varphi _{2}\mid \mathcal{F}_{t})<\min
\{X_{2}(t),Y_{2}(t)\}-\varepsilon )\leq 2\varepsilon .
\end{equation*}%
Because $\lambda _{1}$ is arbitrary, Eq. (\ref{equ zero2}) implies that
\begin{equation*}
{\mathbf{P}}(v_{2}(t)<\min \{X_{2}(t),Y_{2}(t)\}-\varepsilon )\leq
2\varepsilon .
\end{equation*}%
The left-hand side inequality for $i=2$ follows because $\varepsilon $ is
arbitrary.

The proof of the right-hand side-inequality for $i=2$ follows the same
arguments, by using the simple randomized stopping time $\varphi_1(r,\cdot)
= t+r\delta$. Indeed, for every stopping time $\lambda_2$ for player 2 we
then have
\begin{equation*}
\gamma _{2}(\varphi _{1},\lambda _{2}\mid \mathcal{F}_t )= {\mathbf{E}}%
[X_{2}(\varphi _{1})\mathbf{1}_{\{\varphi _{1}<\lambda _{2}\}}
+Y_{2}(\lambda _{2})\mathbf{1}_{\{\varphi _{1}>\lambda _{2}\}}\mid \mathcal{F%
}_t].
\end{equation*}%
The same argument as above, using (\ref{equ zero2b}), delivers the desired
inequality. The proof for $i=1$ is analogous.
\end{proof}

\begin{lemma}
\label{lemma vxy2} For every $t\geq 0$, one has
\begin{eqnarray*}
v_{1}(t) &\leq &\max \{Y_{1}(t),Z_{1}(t)\}\hbox{ on }\Omega , \\
v_{2}(t) &\leq &\max \{X_{2}(t),Z_{2}(t)\}\hbox{ on }\Omega .
\end{eqnarray*}
\end{lemma}

\begin{proof}
We prove the Lemma for $i=1$. Let $\psi_2 = t$: player 2 stops at time $t$.
By (\ref{equ zero1b}),
\begin{equation*}
v_1(t) \leq {\mathrm{ess-sup}}_{\lambda_{1}\geq t}\gamma_1(\lambda_1,\psi_2
\mid \mathcal{F}_t).
\end{equation*}
Because for every (non-randomized) stopping time $\lambda_1$ for player 1, $%
\gamma_1(\lambda_1,\psi_2 \mid \mathcal{F}_t)$ is either $Y_1(t)$ (if $%
\lambda_1 > t$) or $Z_1(t)$ (if $\lambda_1 = t$), the result follows.
\end{proof}

\bigskip

Following Lepeltier and Maingueneau (1984), for every $\eta >0$ let $\mu
_{1}^{\eta }$ and $\mu _{2}^{\eta }$ be the stopping times defined as
follows:
\begin{equation}
\mu _{1}^{\eta }:=\inf \{s\geq 0\colon X_{1}(s)\geq v_{1}(s)-\eta \},
\label{equ mu1}
\end{equation}%
and
\begin{equation}
\mu _{2}^{\eta }:=\inf \{s\geq 0\colon Y_{2}(s)\geq v_{2}(s)-\eta \}.
\label{equ mu2}
\end{equation}

As the following example shows, the stopping times $\mu_{1}^{\eta }$ and $%
\mu _{2}^{\eta }$ may be infinite. Consider the following Dynkin game, where
the payoffs are constants: $X_1 = 0$, $Y_1 = Z_1 = 2$ and $\xi_1 = 1$. Then $%
v_1(t) = 1$ for every $t$, and $\mu^\eta_1 = \infty$, provided $\eta \in
(0,1)$.

Observe that $\mu _{2}^{\eta }\leq \mu _{2}^{\eta ^{\prime }}$ whenever $%
\eta >\eta ^{\prime }$. Moreover, because the processes $X_{1}$, $Y_{2}$, $%
v_{1}$ and $v_{2}$ are right continuous, we have
\begin{equation}
X_{1}(\mu _{1}^{\eta })\geq v_{1}(\mu _{1}^{\eta })-\eta ,
\label{equ martingale1}
\end{equation}%
and
\begin{equation}
Y_{2}(\mu _{2}^{\eta })\geq v_{2}(\mu _{2}^{\eta })-\eta .
\label{equ martingale1a}
\end{equation}%
For every $t < \mu^\eta_1$, by the definition of $\mu^\eta_1$ and  Lemma \ref{lemma vxy},
we have
\[ X_1(t) < v_1(t) - \eta < v_1(t) \leq \max\{X_1(t),Y_1(t)\},\]
and therefore
\begin{equation}  \label{equ martingale2}
Y_{1}(t)> X_{1}(t),\ \ \ \forall t<\mu _{1}^{\eta }.
\end{equation}%
The analogous
inequality for player 2 holds as well.

\begin{lemma}
\label{lemma new1} Let $\varepsilon,\eta > 0$, let $\tau$ be a stopping
time, and let $A \in \mathcal{F}_{\tau}$ satisfy ${\mathbf{P}}(A \setminus
\{\mu^\eta_1 = \infty\}) < \varepsilon$. Then
\begin{equation}  \label{equ 501}
{\mathbf{E}}[v_1(\tau) \mathbf{1}_A] \leq {\mathbf{E}}[\xi_1 \mathbf{1}%
_{A\cap \{\mu^\eta_1 = \infty\}}] + 3\varepsilon + 6\varepsilon/\eta.
\end{equation}
\end{lemma}

\begin{proof}
Let $\psi_2 = \infty$: player 2 never stops. By (\ref{equ zero1b}),
\begin{equation}  \label{equ 400}
v_1(\tau) \leq {\mathrm{ess-sup}}_{ \lambda_{1}\geq
\tau}\gamma_1(\lambda_1,\psi_2 \mid \mathcal{F}_\tau).
\end{equation}
Let $\lambda_1 \geq \tau$ be a stopping time for player 1 that achieves the
supremum in (\ref{equ 400}) up to $\varepsilon$. Let $\lambda^{\prime }_1$
be the following stopping time:

\begin{itemize}
\item On $A \cap \{\lambda_1 < \infty\}$, $\lambda^{\prime }_1$ is an $%
\eta/2 $-optimal stopping time for player 1 in $\Gamma_1(\lambda_1)$.

\item On $A \cap \{\lambda_1 = \infty\}$, $\lambda^{\prime }_1 = \infty$.
\end{itemize}

It follows that
\begin{eqnarray*}
{\mathbf{E}}[v_{1}(\tau )\mathbf{1}_{A}] &\leq &{\mathbf{E}}[\gamma
_{1}(\lambda _{1},\psi _{2}\mid \mathcal{F}_{\tau })\mathbf{1}%
_{A}]+\varepsilon   \notag \\
&=&{\mathbf{E}}[X_{1}(\lambda _{1})\mathbf{1}_{A\cap\{\lambda _{1}<\infty \}}+\xi _{1}\mathbf{1}_{A\cap\{\lambda _{1}=\infty \}}]+\varepsilon  \\
&<&{\mathbf{E}}[(v_{1}(\lambda _{1})-\eta )\mathbf{1}_{A\cap\{\lambda _{1}<\mu
_{1}^{\eta }=\infty \}}+X_{1}(\lambda _{1})\mathbf{1}_{A\cap\{\lambda
_{1}<\infty \}\cap \{\mu _{1}^{\eta }<\infty \}}+\xi _{1}\mathbf{1}%
_{A\cap\{\lambda _{1}=\infty \}}]+\varepsilon   \notag \\
&\leq &{\mathbf{E}}[(v_{1}(\lambda _{1})-\eta )\mathbf{1}_{A\cap\{\lambda
_{1}<\infty \}}+\xi _{1}\mathbf{1}_{A\cap\{\lambda _{1}=\infty \}}]+3\varepsilon  \\
&\leq &{\mathbf{E}}[\gamma _{1}(\lambda _{1}^{\prime },\psi _{2}\mid
\mathcal{F}_{\tau })\mathbf{1}_{A}]-\frac{\eta }{2}{\mathbf{E}}[\mathbf{1}%
_{A\cap\{\lambda _{1}<\infty \}}]+3\varepsilon \\
&\leq &{\mathbf{E}}[\gamma _{1}(\lambda _{1},\psi _{2}\mid \mathcal{F}_{\tau
})\mathbf{1}_{A}]-\frac{\eta }{2}{\mathbf{E}}[\mathbf{1}_{A\cap\{\lambda
_{1}<\infty \}}]+4\varepsilon
\end{eqnarray*}%
where the second inequality holds by the definition of $\mu _{1}^{\eta }$,
the third inequality holds since ${\mathbf{P}}(A\setminus \{\mu _{1}^{\eta
}=\infty \})<\varepsilon $ and since payoffs are bounded by 1, and the last
inequality holds because $\lambda _{1}$ is $\varepsilon $-optimal.

This sequence of inequalities implies that
\begin{equation*}
{\mathbf{P}}(A \cap \{\lambda_1 < \infty\}) \leq 6\varepsilon/\eta,
\end{equation*}
and therefore
\begin{equation*}
{\mathbf{E}}[v_{1}(\tau )\mathbf{1}_{A}] \leq {\mathbf{E}}[\xi_1 \mathbf{1}%
_{A\cap \{\mu^\eta_1 = \infty\}}] + 3\varepsilon + 6\varepsilon/\eta,
\end{equation*}
as desired.
\end{proof}

By Lepeltier and Maingueneau (1984), for each $i=1,2$ the process $v_i$ is a
submartingale up to time $\mu^\eta_i$.

\begin{lemma}
\label{Lemma submartingale} For every $\eta > 0$ the process $%
(v_{1}(t))_{t=0}^{\mu _{1}^{\eta}}$ is a submartingale: for every pair of
finite stopping times $\lambda<\lambda^{\prime }\leq \mu _{1}^{\eta}$ one has $%
v_{1}(\lambda)\leq {\mathbf{E}}[v_{1}(\lambda^{\prime })\mid \mathcal{F}%
_{\lambda}]$ on $\Omega$.
\end{lemma}

Lemma \ref{Lemma submartingale} implies that before time $\sup_{\eta >0}\mu
_{1}^{\eta }$ player $1$ is better off waiting and not stopping. An analogue
statement holds for player 2.

Lemmas \ref{lemma new1} and \ref{Lemma submartingale} deliver the following
result.

\begin{lemma}
\label{lemma exp v} Let $\eta > 0$. For every stopping time $\lambda_1$ that
satisfies $\lambda_1 \leq \mu^\eta_1$ one has
\begin{equation*}
v_1(\lambda_1) \leq {\mathbf{E}}[v_1(\mu^\eta_1)\mathbf{1}_{\{\mu^\eta_1 <
\infty\}} + \xi_1\mathbf{1}_{\{\mu^\eta_1 =\infty\}}\mid\mathcal{F}%
_{\lambda_1}].
\end{equation*}
\end{lemma}

\begin{proof}
Let $\varepsilon> 0$ be arbitrary. By Lemma \ref{Lemma submartingale}, for
every $t \geq 0$ one has
\begin{equation*}
v_1(\lambda_1) \leq {\mathbf{E}}[v_1(\min\{\mu^\eta_1,t\})].
\end{equation*}
Let $t_0$ be sufficiently large such that ${\mathbf{P}}(t_0 \leq \mu^\eta_1
< \infty) < \varepsilon$. By Lemma \ref{lemma new1} with $\tau = t_0$ and $A
= \{t_0 \leq \mu^\eta_1\}$,
\begin{equation*}
{\mathbf{E}}[v_1(t_0) \mathbf{1}_{\{t_0 \leq \mu^\eta_1\}}] \leq {\mathbf{E}}%
[\xi_1 \mathbf{1}_{\{t_0 \leq \mu^\eta_1 = \infty\}}] + 3\varepsilon +
6\varepsilon/\eta.
\end{equation*}
Therefore,
\begin{eqnarray*}
v_1(\lambda_1) &\leq& {\mathbf{E}}[v_1(\min\{\mu^\eta_1,t_0\})] \\
&=& {\mathbf{E}}[v_1(\mu^\eta_1) \mathbf{1}_{\{\mu^\eta_1 < t_0\}} +
v_1(t_0) \mathbf{1}_{\{t_0 \leq \mu^\eta_1\}}] \\
&\leq& {\mathbf{E}}[v_1(\mu^\eta_1) \mathbf{1}_{\{\mu^\eta_1 < \infty\}} +
\xi_1 \mathbf{1}_{\{\mu^\eta_1 = \infty\}}] + 5\varepsilon +
6\varepsilon/\eta.
\end{eqnarray*}
The result follows since $\varepsilon$ is arbitrary.
\end{proof}

\bigskip

The proof of Laraki and Solan
(2005, Section 3.3) delivers the following result, which says
that each player $i$ has a simple randomized $\varepsilon $-optimal stopping
time that is based on $\mu _{i}^{\eta }$, provided $\eta $ is sufficiently
small.

\begin{lemma}
\label{lemma optimal strategies phi_i^0} For every $i=1,2$, every $%
\varepsilon,\eta >0$, and every positive $\mathcal{F}_{\mu _{i}^{\eta }}$%
-measurable function $\delta _{i}$, there exists a simple randomized
stopping time $\varphi _{i}^{\eta }$ with basis $\mu _{i}^{\eta }$ and delay
at most $\delta _{i}$ that satisfies
\begin{equation}
\gamma _{i}(\varphi _{i}^{\eta },\lambda _{3-i}\mid \mathcal{F}_{\mu
_{i}^{\eta }})\geq v_{i}(\mu _{i}^{\eta })-\varepsilon -\eta \hbox{ on }%
\Omega ,  \label{equ 713}
\end{equation}%
for every stopping time $\lambda _{3-i}\geq \mu _{i}^{\eta }$.
\end{lemma}

By Eq. (\ref{equ martingale2}), before time $\mu^\eta_1$ one has $X_1 < Y_1$%
. When $X_{1}(t)\leq Z_{1}(t)\leq Y_{1}(t)$ for every $t,$ a non-randomized $%
\varepsilon $-optimal stopping time exists (Lepeltier and Maingueneau,
1984). Laraki and Solan (2002, Section 4.1) use this observation to conclude
the following.

\begin{lemma}
\label{Lemma purification}If $Z_{i}(t)\in co\{X_{i}(t),Y_{i}(t)\}$ for every
$t \geq 0$ and each $i=1,2$, then the simple randomized stopping time $%
\varphi _{i}^{\eta }$ in Lemma \ref{lemma optimal strategies phi_i^0} can be
taken to be non-randomized (that is, the delay of both players is 0).
\end{lemma}

\section{The Non-Zero-Sum Case}

\label{sec1}

In the present section we prove Theorems \ref{theorem1} and \ref{theorem2}.
Fix $\varepsilon > 0$ once and for all.

Let $\delta_0$ (resp. $\delta_1$, $\delta_2$) be a positive $\mathcal{F}_\tau$%
-measurable function that satisfies the following inequalities for each $%
i\in \{1,2\}$ and for the stopping time $\tau = 0$ (resp. $\tau= \mu
_{1}^{\eta }$, $\tau=\mu _{2}^{\eta }$). Such $\delta_0 $ (resp. $\delta_1$,
$\delta_2$) exists because the processes $(X_{i},Y_{i},v_{i})_{i=1,2}$
are right continuous.
\begin{eqnarray}
{\mathbf{P}}( \sup_{\rho\in [0,\delta_0]}|X_{i}(\tau )-X_{i}(\tau+\rho)| > \varepsilon) &\leq&
\varepsilon,  \label{equ 59} \\
{\mathbf{P}}( \sup_{\rho\in [0,\delta_0]}|Y_{i}(\tau )-Y_{i}(\tau+\rho)| > \varepsilon) &\leq&
\varepsilon,  \label{equ 60} \\
{\mathbf{P}}( \sup_{\rho\in [0,\delta_0]}|v_{i}(\tau )-v_{i}(\tau+\rho)| > \varepsilon) &\leq&
\varepsilon.  \label{equ 61}
\end{eqnarray}

We divide the set $\Omega $ into six $\mathcal{F}_{0}$-measurable
subsets. For each of these subsets we then define a pair of randomized
stopping times $(\varphi _{1}^{\ast },\varphi _{2}^{\ast })$, and we prove
that, when restricted to each set, this pair is a $k\varepsilon $%
-equilibrium, for some $0\leq k\leq 13$.
It will then follow that $(\varphi _{1}^{\ast },\varphi _{2}^{\ast })$,
when viewed as a randomized stopping time on $\Omega$,
is a $78\ep$-equilibrium.
The partition is similar to that in
Laraki, Solan and Vieille (2005), and only the treatment on the last subset
is different.

\bigskip

Denote by $\psi_i(t,\varepsilon)$ an $\varepsilon$-optimal stopping time of
player $i$ in the game $\Gamma_{3-i}(t) $; thus, the randomized stopping
time $\psi_i(t,\varepsilon)$ is a punishment strategy against player $3-i$,
as it ensures that his payoff will not exceed $v_{3-i}(t) + \varepsilon$.

\bigskip \noindent\textbf{Part 1:} The set $A_1 := \{X_1(0) \geq v_1(0)\}
\cap \{X_2(0) \geq Z_2(0)\}$.

We prove that when restricted to the set $A_{1}$, the pair $(\varphi
_{1}^{\ast },\varphi _{2}^{\ast })$ that is defined as follows is a $%
4\varepsilon $-equilibrium:

\begin{itemize}
\item $\varphi^*_1 = 0$: player 1 stops at time 0.

\item $\varphi^*_2 = \psi_2(\delta_0,\varepsilon)$: If player 1 does not
stop before time $\delta_0$, player 2 punishes him in the game $%
\Gamma_1(\delta_0)$ that starts at time $\delta_0$.
\end{itemize}

If no player deviates, the game is stopped by player 1, and the payoff is
\begin{equation*}
\gamma(\varphi^*_1,\varphi^*_2\mid\mathcal{F}_0) = (X_{1}(0),X_{2}(0)) %
\hbox{ on } A_1.
\end{equation*}
We argue that player 2 cannot profit by deviating. Indeed, let $\lambda_2 $
be any non-randomized stopping time of player 2. Then on $A_1$
\begin{equation*}
\gamma_2(\varphi^*_1,\lambda_2\mid\mathcal{F}_0) = Z_2(0)\mathbf{1}_{A_1
\cap \{\lambda_2 = 0\}} + X_2(0)\mathbf{1}_{A_1 \cap \{\lambda_2 > 0\}} \leq
X_2(0) = \gamma_2(\varphi^*_1,\varphi^*_2\mid\mathcal{F}_0),
\end{equation*}
and the claim follows.

We now argue that on $A_1$ player 1 cannot profit more than $4\varepsilon$
by deviating from $\varphi^*_1$. Let $\lambda_1$ be any non-randomized
stopping time of player 1. Then by the definition of $\varphi^*_2$, on $A_1$
\begin{equation*}
\gamma_1(\lambda_1,\varphi^*_2\mid\mathcal{F}_0) \leq {\mathbf{E}}%
[X_1(\lambda_1) \mathbf{1}_{\{\lambda_1 < \delta_0\}} + (v_1(\delta_0) +
\varepsilon) \mathbf{1}_{\{\delta_0 \leq \lambda_1\}}\mid \mathcal{F}_0].
\end{equation*}
By (\ref{equ 59}), (\ref{equ 61}), and since $X_1(0) \geq v_1(0)$ on $A_1$,
it follows that on $A_1$
\begin{equation*}
{\mathbf{P}}( \gamma_1(\lambda_1,\varphi^*_2\mid\mathcal{F}_0) > {\mathbf{E}}%
[X_1(0) \mathbf{1}_{\{\lambda_1 < \delta_0\}} + (X_1(0) + \varepsilon)
\mathbf{1}_{\{\delta_0 \leq \lambda_1\}}\mid \mathcal{F}_0] + \varepsilon)
\leq 2\varepsilon.
\end{equation*}
Since $\gamma_1(\varphi^*_1,\varphi^*_2\mid\mathcal{F}_0) = X_1(0)$ on $A_1$
it follows that
\begin{equation}  \label{equ 84}
{\mathbf{P}}( A_1 \cap \{\gamma_1(\lambda_1,\varphi^*_2\mid\mathcal{F}_0) >
\gamma_1(\varphi^*_1,\varphi^*_2\mid\mathcal{F}_0)+ 2\varepsilon\}) \leq
2\varepsilon,
\end{equation}
and the desired results follows.

\bigskip

\bigskip \noindent \textbf{Part 2:} The set $A_{2}:=\{Z_{2}(0)>X_{2}(0)\}%
\cap \{Z_{1}(0)\geq Y_{1}(0)\}$.

We prove that when restricted to the set $A_2$, the pair $%
(\varphi_1^*,\varphi^*_2)$ that is defined as follows is a 0-equilibrium:

\begin{itemize}
\item $\varphi^*_1 = 0$: player 1 stops at time 0.

\item $\varphi^*_2 = 0$: player 2 stops at time 0.
\end{itemize}

If no player deviates, both players stop at time 0, and the payoff is
\begin{equation*}
\gamma(\varphi^*_1,\varphi^*_2\mid\mathcal{F}_0) = (Z_{1}(0),Z_{2}(0)) %
\hbox{ on } A_2.
\end{equation*}
To see that player 1 cannot profit by deviating, fix an arbitrary
non-randomized stopping time $\lambda_1$ for player 1. On $A_2$ one has
\begin{equation}  \label{equ 84b}
\gamma_1(\lambda_1,\varphi^*_2\mid\mathcal{F}_0) = Z_1(0) \mathbf{1}%
_{\{\lambda_1 = 0\}} + Y_1(0) \mathbf{1}_{\{\lambda_1 > 0\}} \leq Z_1(0) =
\gamma_1(\varphi^*_1,\varphi^*_2\mid\mathcal{F}_0),
\end{equation}
as desired. A symmetric argument shows that player 2 cannot profit by
deviating either.

\bigskip

\bigskip \noindent\textbf{Part 3:} The set $A_3 := \{Y_1(0) > Z_1(0)\} \cap
\{Y_2(0) \geq v_2(0)\}$.

The case of the set $A_{3}$ is analogous to Part 1: when restricted to $A_3$%
, the pair of randomized stopping times in which player 2 stops at time 0,
and player 1 plays an $\varepsilon $-optimal stopping time $%
\psi_1(\delta_0,\varepsilon)$ in the game $\Gamma _{2}(\delta_0 )$, is a $%
4\varepsilon$-equilibrium.

\bigskip

\bigskip \noindent \textbf{Part 4:} The set $A_{4}:=\{X_{1}(0)\geq
v_{1}(0)\}\cap \{X_{2}(0)>Y_{2}(0)\}$.

We prove that when restricted to the set $A_{4}$, the pair $(\varphi
_{1}^{\ast },\varphi _{2}^{\ast })$ that is defined as follows is a $%
6\varepsilon $-equilibrium:

\begin{itemize}
\item $\varphi _{1}^{\ast }(r,\cdot )=r\delta _{0}$: player 1 stops at a
random time between time $0$ and time $\delta _{0}$.

\item $\varphi _{2}^{\ast }=\psi _{2}(\delta _{0},\varepsilon )$: If player
1 does not stop before time $\delta _{0}$, player 2 punishes him in the game
$\Gamma _{1}(\delta _{0})$ that starts at time $\delta _{0}$.
\end{itemize}

If no player deviates, the game is stopped by player 1 before time $\delta_0
$, and by (\ref{equ 59}) the payoff is within $2\varepsilon $ of $%
(X_{1}(0),X_{2}(0))$:
\begin{equation}  \label{equ 84d}
{\mathbf{P}}(A_4 \cap\{|\gamma _{i}(\varphi _{1}^{\ast },\varphi _{2}^{\ast
})-X_{i}(0)|>\varepsilon\})\leq\varepsilon.
\end{equation}%
The same argument\footnote{%
The additional $\varepsilon$ arises because in Part 1 we had $%
\gamma_1(\varphi^*_1,\varphi^*_2) = X_1(0)$, whereas in Part 4 we have ${%
\mathbf{P}}(A_4\cap\{\gamma_1(\varphi^*_1,\varphi^*_2) <
X_1(0)-\varepsilon\}) \leq \varepsilon$.} as in Part 1 shows that
\begin{equation}  \label{equ 84c}
{\mathbf{P}}( A_4 \cap \{\gamma_1(\lambda_1,\varphi^*_2\mid\mathcal{F}_0) >
\gamma_1(\varphi^*_1,\varphi^*_2\mid\mathcal{F}_0)+ 3\varepsilon\}) \leq
3\varepsilon.
\end{equation}
It follows that player 1 cannot profit more than $6\varepsilon$ by deviating
from $\varphi^*_1$.

We now argue that player 2 cannot profit more than $5\varepsilon$ by
deviating from $\varphi^*_2$. Fix a non-randomized stopping time $\lambda
_{2}$ for player 2. On $A_4$ we have $\varphi^*_1 \leq \delta_0$, and ${%
\mathbf{P}}(A_4 \cap\{ \varphi^*_1 = \lambda_2\}) = 0$, and therefore
\begin{equation*}
\gamma _{2}(\varphi _{1}^{\ast },\lambda _{2}) ={\mathbf{E}}[X_{2}(\varphi
_{1}^{\ast })\mathbf{1}_{\{\varphi _{1}^{\ast }<\lambda
_{2}\}}+Y_{2}(\lambda _{2})\mathbf{1}_{\{\lambda _{2}<\varphi _{1}^{\ast
}\}} \mid \mathcal{F}_0] \hbox{ on } A_4.
\end{equation*}
By (\ref{equ 59}) and (\ref{equ 60}),
\begin{equation*}
{\mathbf{P}}( \gamma _{2}(\varphi _{1}^{\ast },\lambda _{2}) > {\mathbf{E}}%
[(X_2(0) + \varepsilon)\mathbf{1}_{\{\varphi _{1}^{\ast }<\lambda
_{2}\}}+(Y_{2}(0)+\varepsilon)\mathbf{1}_{\{\lambda _{2}<\varphi _{1}^{\ast
}\}} \mid \mathcal{F}_0]) \leq 2\varepsilon.
\end{equation*}
Because $X_2(0) > Y_2(0)$ on $A_4$ we have
\begin{equation*}
{\mathbf{P}}( \gamma _{2}(\varphi _{1}^{\ast },\lambda _{2}) > X_2(0) +
\varepsilon) \leq 2\varepsilon.
\end{equation*}
Together with (\ref{equ 84d}) we deduce that
\begin{equation*}
{\mathbf{P}}( \gamma _{2}(\varphi _{1}^{\ast },\lambda _{2}) > \gamma
_{2}(\varphi _{1}^{\ast },\varphi^* _{2}) + 2\varepsilon) \leq 3\varepsilon,
\end{equation*}
and the claim follows.

\bigskip \noindent \textbf{Part 5:} The set $A_{5}:=\{X_{1}(0)\geq
v_{1}(0)\}\setminus (A_{1}\cup A_{2}\cup A_{3}\cup A_{4})$.

We claim that ${\mathbf{P}}(A_5) = 0$. Since $X_{1}(0)\geq v_{1}(0)$ on $A_5$%
, and since $A_5 \cap A_1 = \emptyset$, it follows that $X_{2}(0)<Z_{2}(0)$
on $A_5$. Since $A_5 \cap A_2 = \emptyset$, it follows that $%
Z_{1}(0)<Y_{1}(0)$ on $A_5$. Since $A_5\cap A_3 = \emptyset$, it follows
that $Y_{2}(0)<v_{2}(0)$ on $A_5$. Since $A_5 \cap A_4 = \emptyset$, it
follows that $Y_{2}(0)\geq X_{2}(0)$ on $A_5$. Lemma \ref{lemma vxy} then
implies that
\begin{equation*}
Y_2(0) < v_{2}(0)\leq \max \{X_{2}(0),Y_{2}(0)\} = Y_2(0) \hbox{ on } A_5,
\end{equation*}
which in turn implies that ${\mathbf{P}}(A_5) = 0$, as claimed.

\bigskip

The union $A_{1}\cup A_{2}\cup A_{3}\cup A_{4}\cup A_{5}$ includes the set $%
\{X_{1}(0)\geq v_{1}(0)\}$. Thus, when restricted to this set, the game has
a $7\varepsilon $-equilibrium. By symmetric arguments, a $6\varepsilon $%
-equilibrium exists on the set $\{Y_{2}(0)\geq v_{2}(0)\}$. We now
construct
a $13\varepsilon$-equilibrium on the remaining set, $%
\{X_{1}(0)<v_{1}(0)\}\cap \{Y_{2}(0)<v_{2}(0)\}$.

\bigskip

\bigskip \noindent\textbf{Part 6:} The set $A_6 := \{X_1(0) < v_1(0)\} \cap
\{Y_2(0) < v_2(0)\}$.

Fix $\eta > 0$, and for each $i\in \{1,2\}$ let $\varphi _{i}^{\eta }$ be a
simple randomized stopping time with basis $\mu^\eta_i$ and delay at most $%
\delta_i$ that satisfies Eq. (\ref{equ 713}) for every stopping time $%
\lambda_{3-i} \geq \mu^\eta_i$ (see Lemma \ref{lemma optimal strategies
phi_i^0}). Let $\psi_1(\mu^\eta_2+\delta_2,\varepsilon)$ (resp. $%
\psi_2(\mu^\eta_1+\delta_1,\varepsilon)$) be a simple randomized $%
\varepsilon $-optimal stopping time for player 1 in the game $\Gamma
_{2}(\mu _{2}^{\eta }+\delta_2 )$ (resp. in the game $\Gamma _{1}(\mu
_{1}^{\eta }+\delta_1 )$); that is, a stopping time that achieves the
infimum in (\ref{equ zero2b}) up to $\varepsilon $, for $t=\mu
_{2}^{\eta}+\delta_2$ (resp. the infimum in (\ref{equ zero1b}) up to $%
\varepsilon $, for $t=\mu _{1}^{\eta }+\delta_1 $).

Set $\mu ^{\eta }=\min \{\mu _{1}^{\eta },\mu _{2}^{\eta }\}$. We further
divide $A_{6}$ into six $\mathcal{F}_{\mu^\eta}$-measurable subsets; the
definition of $(\varphi _{1}^{\ast },\varphi _{2}^{\ast })$ is different in
each subset, and is given in the second and third columns of Table 1. Under $%
(\varphi _{1}^{\ast },\varphi _{2}^{\ast })$ the game will be stopped at
time $\mu ^{\eta }$ or during a short interval after time $\mu ^{\eta }$, if
$\mu ^{\eta }<\infty $, and will not be stopped if $\mu ^{\eta }=\infty $.

\begin{equation*}
\begin{array}{|l|l|l|l|}
\hbox{Subset} & \varphi _{1}^{\ast } & \varphi _{2}^{\ast } & \gamma
_{1}(\varphi _{1}^{\ast },\varphi _{2}^{\ast }) \\ \hline\hline
A_{61}:=A_{6}\cap \left\{ {\mu _{1}^{\eta }<\mu _{2}^{\eta }}\right\}  &
\varphi _{1}^{\eta } & \psi _{2}(\mu ^{\eta }+\delta _{1}) & \geq X_{1}(\mu
^{\eta })-2\varepsilon  \\ \hline
A_{62}:=A_{6}\cap \left\{ \mu _{2}^{\eta }<\mu _{1}^{\eta }\right\}  & \psi
_{1}(\mu ^{\eta }+\delta _{2}) & \varphi _{2}^{\eta } & \geq Y_{1}(\mu
^{\eta })-2\varepsilon  \\ \hline
A_{63}:=A_{6}\cap \left\{ {\mu _{1}^{\eta }=\mu _{2}^{\eta }=\infty }%
\right\}  & \infty  & \infty  & =\xi _{1} \\ \hline
A_{64}:=A_{6}\cap \left\{ {\mu _{1}^{\eta }=\mu _{2}^{\eta }<\infty }%
\right\} \cap \{Z_{1}(\mu _{1}^{\eta })<Y_{1}(\mu _{1}^{\eta })\} & \psi
_{1}(\mu _{1}^{\eta }+\delta _{2},\varepsilon ) & \mu ^{\eta } & =Y_{1}(\mu
^{\eta }) \\ \hline
A_{65}:=A_{6}\cap \left\{ {\mu _{1}^{\eta }=\mu _{2}^{\eta }<\infty }%
\right\} \cap \{Z_{2}(\mu _{1}^{\eta })<X_{2}(\mu _{1}^{\eta })\} & \mu
^{\eta } & \psi _{2}(\mu ^{\eta }+\delta _{1},\varepsilon ) & =X_{1}(\mu
^{\eta }) \\ \hline
A_{66}:=A_{6}\cap \left\{ {\mu _{1}^{\eta }=\mu _{2}^{\eta }<\infty }%
\right\} \cap \{Y_{1}(\mu _{1}^{\eta })\leq Z_{1}(\mu _{1}^{\eta })\} &  &
&  \\
\ \ \ \ \ \ \ \ \ \ \ \ \ \ \cap \{X_{2}(\mu ^{\eta })\leq Z_{2}(\mu ^{\eta
})\} & \mu ^{\eta } & \mu ^{\eta } & =Z_{1}(\mu ^{\eta }) \\ \hline
\end{array}%
\end{equation*}

\centerline{Table 1: The randomized stopping times $(\varphi _{1}^{\ast },\varphi _{2}^{\ast })$ on $A_6$,
with the payoff to player 1.}

\bigskip

We argue that when restricted to $A_{6}$, the pair $(\varphi _{1}^{\ast
},\varphi _{2}^{\ast })$ is a $13\varepsilon$-equilibrium. Note that
the roles of the two players in the definition of $(\varphi _{1}^{\ast
},\varphi _{2}^{\ast })$ are symmetric: $\varphi _{1}^{\ast }=\varphi
_{2}^{\ast }$ on $A_{63}$ and $A_{66}$, and the role of player 1 (resp.
player 2) in $A_{61}$ and $A_{64}$ is similar to the role of player 2 (resp.
player 1) in $A_{62}$ and $A_{65}$. To prove that $(\varphi _{1}^{\ast
},\varphi _{2}^{\ast })$ is a $13\varepsilon $-equilibrium it is
therefore sufficient to prove that the probability that player 1 can profit
more than $3\varepsilon$ by deviating from $\varphi _{1}^{\ast }$ is
at most $10\varepsilon $.

We start by bounding the payoff $\gamma_1(\varphi _{1}^{\ast },\varphi
_{2}^{\ast } \mid \mathcal{F}_{\mu^\eta})$ (the bound that we derive appears
on the right-most column in Table 1), and by showing that
\begin{equation}  \label{equ403}
\gamma_1(\varphi _{1}^{\ast },\varphi _{2}^{\ast } \mid \mathcal{F}%
_{\mu^\eta}) \geq v_1(\mu^\eta) - 3\varepsilon-\eta \hbox{ on } A_6
\setminus A_{63}.
\end{equation}
We prove this in turn on each of the sets $A_{61},\dots,A_{66}$:

\begin{itemize}
\item On $A_{61}$ we have $\mu^\eta = \mu^\eta_1$, and the game is stopped
by player 1 between times $\mu^\eta$ and $\mu^\eta+\delta_1$, so that by (%
\ref{equ 59}) we have
\begin{equation}  \label{equ 561}
{\mathbf{P}}(A_{61} \cap \{\gamma_1(\varphi _{1}^{\ast },\varphi _{2}^{\ast
} \mid \mathcal{F}_{\mu^\eta}) <X_1(\mu^\eta) - \varepsilon\}) \leq
\varepsilon.
\end{equation}
By (\ref{equ martingale1}) we have $X_1(\mu^\eta) \geq v_1(\mu^\eta)-\eta$,
and therefore (\ref{equ403}) holds on $A_{61}$.

\item On $A_{62}$ we have $\mu^\eta = \mu^\eta_2$, and the game is stopped
by player 2 between times $\mu^\eta$ and $\mu^\eta+\delta_2$, so that by (%
\ref{equ 60}) we have
\begin{equation}  \label{equ 57a}
{\mathbf{P}}(A_{62} \cap \{\gamma_1(\varphi _{1}^{\ast },\varphi _{2}^{\ast
} \mid \mathcal{F}_{\mu^\eta}) <Y_1(\mu^\eta) - \varepsilon\}) \leq
\varepsilon.
\end{equation}
By (\ref{equ martingale2}) we have $X_1(\mu^\eta) < Y_1(\mu^\eta)$ on $A_{62}
$, so that by Lemma \ref{lemma vxy} we have $Y_1(\mu^\eta) \geq v_1(\mu^\eta)
$. It follows that (\ref{equ403}) holds on $A_{62}$.

\item On $A_{63}$ no player ever stops, and therefore $\gamma_1(\varphi
_{1}^{\ast },\varphi _{2}^{\ast } \mid \mathcal{F}_{\mu^\eta}) = \xi_1$.

\item On $A_{64}$ player 2 stops at time $\mu^{\eta }$, and therefore $%
\gamma _{1}(\varphi _{1}^{\ast },\varphi _{2}^{\ast } \mid \mathcal{F}%
_{\mu^\eta})=Y_{1}(\mu^{\eta })$. By Lemma \ref{lemma vxy2}, on $A_{64}$ we
have
\begin{equation*}
v_{1}(\mu^{\eta })\leq \max \{Y_{1}(\mu^{\eta }),Z_{1}(\mu^{\eta
})\}=Y_{1}(\mu^{\eta }),
\end{equation*}%
and therefore (\ref{equ403}) holds on $A_{64}$.

\item On $A_{65}$ player 1 stops at time $\mu^{\eta }$, and therefore $%
\gamma _{1}(\varphi _{1}^{\ast },\varphi _{2}^{\ast } \mid \mathcal{F}%
_{\mu^\eta})=X_{1}(\mu^{\eta })$. By (\ref{equ martingale1}) we have $%
X_{1}(\mu^{\eta })\geq v_{1}(\mu^{\eta })-\eta $, and therefore (\ref{equ403}%
) holds on $A_{65} $.

\item On $A_{66}$ both players stop at time $\mu^{\eta }$, and therefore $%
\gamma _{1}(\varphi _{1}^{\ast },\varphi _{2}^{\ast } \mid \mathcal{F}%
_{\mu^\eta})=Z_{1}(\mu^{\eta })$. By Lemma \ref{lemma vxy2} on this set we
have
\begin{equation*}
v_{1}(\mu^{\eta })\leq \max \{Y_{1}(\mu^{\eta }),Z_{1}(\mu^{\eta
})\}=Z_{1}(\mu^{\eta }),
\end{equation*}%
and therefore (\ref{equ403}) holds on $A_{66}$.
\end{itemize}

Fix a stopping time $\lambda_1$ for player 1. To complete the proof of
Theorem \ref{theorem1} we prove that
\begin{equation*}
{\mathbf{P}}(A_6 \cap \{\gamma_1(\lambda_1,\varphi^*_2) >
\gamma_1(\varphi^*_1,\varphi^*_2) + 3\varepsilon\}) \leq 10\varepsilon.
\end{equation*}

\begin{itemize}
\item On the set $A_6 \cap \{\lambda_1 < \mu^\eta\}$ we have by the
definition of $\mu^\eta_1$, since $\mu^\eta \leq \mu^\eta_1$, by Lemma %
\ref{lemma exp v}, and by (\ref{equ403}),
\begin{eqnarray}  \label{equ 490}
\gamma_1(\lambda_1,\varphi^*_2 \mid \mathcal{F}_{\lambda_1}) &=&
X_1(\lambda_1)  \notag \\
&<& v_1(\lambda_1) - \eta \\
&\leq& {\mathbf{E}}[v_1(\mu^\eta) \mathbf{1}_{A_6 \cap \{\lambda_1 <
\mu^\eta < \infty\}} + \xi_1 \mathbf{1}_{A_6 \cap \{\lambda_1 < \mu^\eta =
\infty\}}\mid \mathcal{F}_{\lambda_1}] - \eta  \notag \\
&\leq& \gamma_1(\varphi^*_1,\varphi^*_2 \mid \mathcal{F}_{\lambda_1}) +3\ep,
\notag
\end{eqnarray}
where the last inequality holds by (\ref{equ403}) and because the payoff of
player 1 on $A_{63}$ is $\xi_1$.

\item On the set $A_{61} \cap \{ \mu^\eta \leq \lambda_1\}$ we have by the
definition of $\varphi^*_2$
\begin{equation*}
\gamma_1(\lambda_1,\varphi^*_2 \mid \mathcal{F}_{\mu^\eta}) = {\mathbf{E}}%
[X_1(\lambda_1) \mathbf{1}_{\{\lambda_1 \leq \mu^\eta + \delta_1\}} +
(v_1(\mu^\eta+\delta_1)+\varepsilon)\mathbf{1}_{\{\mu^\eta +
\delta_1<\lambda_1\}} \mid \mathcal{F}_{\mu^\eta}].
\end{equation*}
By (\ref{equ 59}), (\ref{equ 61}) and (\ref{equ martingale1}) we have
\begin{equation*}
{\mathbf{P}}( \gamma_1(\lambda_1,\varphi^*_2 \mid \mathcal{F}_{\mu^\eta}) >
X_1(\mu^\eta) + 2\varepsilon )\leq 2\varepsilon.
\end{equation*}
By (\ref{equ 561}) we deduce that
\begin{equation}  \label{equ 601}
{\mathbf{P}}( \gamma_1(\lambda_1,\varphi^*_2 \mid \mathcal{F}_{\mu^\eta}) >
\gamma_1(\varphi^*_1,\varphi^*_2 \mid \mathcal{F}_{\mu^\eta}) + 3\varepsilon
)\leq 3\varepsilon.
\end{equation}

\item On the set $A_{62} \cap \{\mu^\eta \leq \lambda_1\}$ we have by the
definition of $\varphi^*_2$
\begin{equation*}
\gamma_1(\lambda_1,\varphi^*_2 \mid \mathcal{F}_{\mu^\eta}) = {\mathbf{E}}%
[X_1(\lambda_1) \mathbf{1}_{\{\mu^\eta \leq \lambda_1 < \varphi^*_2\}} +
Y_1(\varphi^*_2) \mathbf{1}_{\{\varphi^*_2 \leq \lambda_1\}} \mid \mathcal{F}%
_{\mu^\eta}].
\end{equation*}
By (\ref{equ 59}), (\ref{equ 60}), since $\mu^\eta_2 < \mu^\eta_1$ on $A_{62}
$, and by (\ref{equ martingale2}),
\begin{equation*}
{\mathbf{P}}( \gamma_1(\lambda_1,\varphi^*_2 \mid \mathcal{F}_{\mu^\eta}) > {%
\mathbf{E}}[(Y_1(\mu^\eta)+\varepsilon) \mathbf{1}_{\{\mu^\eta \leq
\lambda_1 < \varphi^*_2\}} + (Y_1(\mu^\eta) +\varepsilon)\mathbf{1}%
_{\{\varphi^*_2 \leq \lambda_1\}} \mid \mathcal{F}_{\mu^\eta}]) \leq
2\varepsilon.
\end{equation*}
By (\ref{equ 57a}) we deduce that
\begin{equation}  \label{equ 602}
{\mathbf{P}}( A_{62} \cap \{\gamma_1(\lambda_1,\varphi^*_2 \mid \mathcal{F}%
_{\mu^\eta}) > \gamma_1(\varphi^*_1,\varphi^*_2 \mid \mathcal{F}_{\mu^\eta})
+ 2\varepsilon\}) \leq 3\varepsilon.
\end{equation}

\item On the set $A_{63} \cap \{\mu^\eta_2 \leq \lambda_1\}$ we have $%
\mu^\eta = \lambda_1 = \infty$, so that
\begin{equation}  \label{equ 341}
\gamma_1(\varphi^*_1,\varphi^*_2\mid\mathcal{F}_{\mu^\eta}) = \xi_1 =
\gamma_1(\lambda_1,\varphi^*_2\mid\mathcal{F}_{\mu^\eta}) \hbox{ on } A_{63}
\cap \{\mu^\eta_2 \leq \lambda_1\}.
\end{equation}

\item On the set $A_{64} \cap \{\mu^\eta \leq \lambda_1\}$ we have
\begin{eqnarray}  \label{equ 603}
\gamma_1(\lambda_1,\varphi^*_2 \mid \mathcal{F}_{\mu^\eta}) &=& {\mathbf{E}}%
[Z_1(\mu^\eta) \mathbf{1}_{\{\lambda_1 = \mu^\eta\}} + Y_1(\mu^\eta) \mathbf{%
1}_{\{\mu^\eta<\lambda_1\}}] \\
&\leq& Y_1(\mu^\eta) = \gamma_1(\varphi^*_1,\varphi^*_2 \mid \mathcal{F}%
_{\mu^\eta}).  \notag
\end{eqnarray}

\item On the set $A_{65} \cap \{\mu^\eta \leq \lambda_1\}$ we have by the
definition of $\varphi^*_2$
\begin{equation*}
\gamma_1(\lambda_1,\varphi^*_2 \mid \mathcal{F}_{\mu^\eta}) = {\mathbf{E}}%
[X_1(\lambda_1) \mathbf{1}_{\{\mu^\eta \leq \lambda_1 < \mu^\eta +
\delta_1\}} + (v_1(\mu^\eta + \delta_1)+\ep) \mathbf{1}_{\{\mu^\eta +
\delta_1\leq\lambda_1\}} \mid \mathcal{F}_{\mu^\eta}].
\end{equation*}
By (\ref{equ 59}), (\ref{equ 61}) and (\ref{equ martingale1}) we have
\begin{equation*}
{\mathbf{P}}( A_{65} \cap \{\mu^\eta \leq \lambda_1\} \cap
\{\gamma_1(\lambda_1,\varphi^*_2 \mid \mathcal{F}_{\mu^\eta}) >
X_1(\mu^\eta)+2\varepsilon \}) \leq 2\varepsilon.
\end{equation*}
Because $\gamma_1(\varphi^*_1,\varphi^*_2 \mid \mathcal{F}_{\mu^\eta}) =X_1(\mu^\eta)$ on $A_{65}$, we obtain
\begin{eqnarray}  \label{equ 604}
{\mathbf{P}}( A_{65} \cap \{\mu^\eta \leq \lambda_1\} \cap
\{\gamma_1(\lambda_1,\varphi^*_2 \mid \mathcal{F}_{\mu^\eta}) >
\gamma_1(\varphi^*_1,\varphi^*_2 \mid \mathcal{F}_{\mu^\eta})+2\varepsilon
\}) \leq 2\varepsilon.
\end{eqnarray}

\item On the set $A_{66} \cap \{\mu^\eta \leq \lambda_1\}$ we have
\begin{eqnarray}
\gamma_1(\lambda_1,\varphi^*_2 \mid \mathcal{F}_{\mu^\eta}) &=&
Z_1(\mu^\eta) \mathbf{1}_{\{\lambda_1 = \mu^\eta\}} + Y_1(\mu^\eta) \mathbf{1%
}_{\{\mu^\eta<\lambda_1\}}  \notag \\
&\leq& Z_1(\mu^\eta) = \gamma_1(\varphi^*_1,\varphi^*_2 \mid \mathcal{F}%
_{\mu^\eta}).  \label{equ 605}
\end{eqnarray}
\end{itemize}

From (\ref{equ 601}), (\ref{equ 602}), (\ref{equ 341}), (\ref{equ 603}), (%
\ref{equ 604}) and (\ref{equ 605}) we deduce that on $A_6 \cap \{\mu^\eta
\leq \lambda_1\} $
\begin{eqnarray}  \label{equ 491}
{\mathbf{P}}( \gamma_1(\lambda_1,\varphi^*_2 \mid \mu^\eta) >
\gamma_1(\varphi^*_1,\varphi^*_2 \mid \mu^\eta) + 3\varepsilon ) \leq
10\varepsilon.
\end{eqnarray}
Because (\ref{equ 490}) and (\ref{equ 491}) hold for every stopping time $%
\lambda_1$ for player 1, it follows that $(\varphi^*_1,\varphi^*_2)$ is a $%
13\varepsilon$-equilibrium on $A_6$, as desired.

\bigskip

\begin{proof}[Proof of Theorem \protect\ref{theorem2}]
To prove that if $Z_{1}(t)\in co\{X_{1}(t),Y_{1}(t)\}$ and $Z_{2}(t)\in
co\{X_{2}(t),Y_{2}(t)\} $ for every $t \geq 0$, then there is a pair of
non-randomized stopping times that form an $\varepsilon $-equilibrium, we
are going to check where randomized stopping times were used in the proof of
Theorem \ref{theorem1}, and we will see how in each case one can use
non-randomized stopping times instead of randomized stopping times.

\begin{enumerate}
\item In Part 1 (and in the analogue part 3) we used a punishment strategy $%
\psi_1(\delta_0,\varepsilon)$ that in general is a non-randomized stopping
time. However, by Lemma \ref{Lemma purification}, when $Z_{2}(t)\in
co\{X_{2}(t),Y_{2}(t)\} $ for every $t \geq 0$, this randomized stopping
time can be taken to be non-randomized.

\item In Part 4 we used, in addition to the punishment strategy $\psi
_{2}(\delta _{0},\varepsilon )$, a simple randomized stopping time for
player 1. The set that we were concerned with in part 4 was the set $%
A_{4}:=\{X_{1}(0)\geq v_{1}(0)\}\cap \{X_{2}(0)>Y_{2}(0)\}.$ Because $%
Z_{2}(0)\in co\{X_{2}(0),Y_{2}(0)\}$
\begin{equation*}
X_{2}(0)\geq Z_{2}(0)\geq Y_{2}(0).
\end{equation*}%
But then the following pair of non-randomized stoping times is a $%
3\varepsilon $-equilibrium when restricted to $A_{4}$:

\begin{itemize}
\item $\varphi _{1}^{\ast }:=0$: player 1 stops at time 0.

\item $\varphi _{2}^{\ast }:=\psi _{2}(\delta _{0},\varepsilon )$: if player
1 does not stop before time $\delta _{0}$, player 2 punishes him (with a
non-randomized stopping time; see first item) in the game $\Gamma
_{1}(\delta _{0})$.
\end{itemize}

\item In Part 6 randomization was used both for punishment (on $A_{61}$, $%
A_{62}$, $A_{64}$ and $A_{65}$) and for stopping (on $A_{61}$ and $A_{62}$).
As mentioned above, under the assumptions of Theorem \ref{theorem2}, for
punishment one can use non-randomized stopping times. We now argue that one
can modify the definition of $(\varphi^*_1,\varphi^*_2)$ on $A_{61}$ and $%
A_{62}$ so as to obtain a non-randomized equilibrium. Because of the
symmetry between $A_{61}$ and $A_{62}$, we show how to modify the
construction only on $A_{61}$.

On $A_{61}$ we have $\mu _{1}^{\eta }<\mu _{2}^{\eta }$, so that by (\ref%
{equ martingale2}) we have $Y_{2}(\mu _{1}^{\eta })<X_{2}(\mu _{2}^{\eta })$%
. Because $Z_{2}(\mu _{1}^{\eta })\in co\{X_{2}(\mu _{1}^{\eta }),Y_{2}(\mu
_{1}^{\eta })\}$ it follows that $Y_{2}(\mu _{1}^{\eta })\leq Z_{2}(\mu
_{1}^{\eta })\leq X_{2}(\mu _{2}^{\eta })$. But then the following pair of
non-randomized stoping times is a $3\varepsilon $-equilibrium on $A_{61}$:

\begin{itemize}
\item $\varphi _{1}^{\ast }:=\mu^\eta_1$: player 1 stops at time $\mu^\eta_1$%
.

\item $\varphi _{2}^{\ast }:=\psi _{2}(\mu^\eta_1 + \delta _{1},\varepsilon
) $: if player 1 does not stop before time $\mu^\eta_1 + \delta _{1}$,
player 2 punishes him (with a non-randomized stopping time; see first item)
in the game $\Gamma _{1}(\mu^\eta_1 + \delta _{1})$.
\end{itemize}
\end{enumerate}
\end{proof}

\appendix

\section{The result of Laraki and Solan (2005)}
\label{appendix}

As mentioned before, Laraki and Solan (2005) proved Theorem \ref{theorem1} for two-player zero-sum Dynkin games.
We need the stronger version that is stated in Lemma \ref{lemma8}, where the payoff is conditioned on the $\sigma$-algebra $\calF_t$.
It turns out that the arguments used by Laraki and Solan (2005) prove this case as well, when one uses the following Lemma
instead of Lemma 4 in Laraki and Solan (2005).

\begin{lemma}
\label{lemma new} Let $X$ be a right-continuous process. For every stopping
time $\lambda$ and every positive $\mathcal{F}_\lambda$-measurable function $%
\varepsilon$ there is a positive $\mathcal{F}_\lambda$-measurable and
bounded function $\delta$ such that:
\begin{equation}  \label{equ91}
|X(\lambda) - {\mathbf{E}}[X(\rho)\mid \mathcal{F}_\lambda]| \leq
\varepsilon,
\end{equation}
for every stopping time $\rho$ that satisfies $\lambda \leq \rho \leq
\lambda+\delta$.
\end{lemma}

\begin{proof}
Because the process $X$ is right continuous, the function $w \mapsto {%
\mathbf{E}}[X(\lambda+w)\mid\mathcal{F}_\lambda]$ is right-continuous at $%
w=0 $ on $\Omega$, and it is equal to $X(\lambda)$ at $w=0$. By defining
\begin{equation*}
\delta^{\prime }= \frac{1}{2}\sup\{w > 0 \colon |X(\lambda) - {\mathbf{E}}%
[X(\lambda+w)\mid \mathcal{F}_\lambda]| \leq \varepsilon\},
\end{equation*}
we obtain a positive $\mathcal{F}_\lambda$-measurable function such that
(\ref{equ91}) is satisfied for every stopping time $\rho$, $\lambda \leq \rho
\leq \lambda+ \delta^{\prime }$. The proof of the Lemma is complete by
setting $\delta = \min\{\delta^{\prime },1\}$.
\end{proof}
\bigskip

This Lemma can also be used to adapt the
proof of Proposition 7 in Laraki and Solan (2005)
in order to prove Lemma \ref{lemma zero sum}, which states that the value process is right continuous.

One can use Lemma \ref{lemma new} to improve some of the bounds given in Section \ref{sec1}.
We chose not to use this Lemma in the paper, so as to unify the arguments given for the various bounds.

\end{document}